\documentclass[final,leqno,onefignum,onetabnum]{siamltex}

\usepackage{subfigure,upgreek,dsfont,braket,amssymb,amsmath}
\usepackage[small,bf]{caption}
\usepackage{booktabs,dsfont}
\usepackage{algorithm}
\usepackage{float}
\usepackage{graphicx}
\usepackage{algorithmicx}
\usepackage{color,float}

\usepackage{amssymb, amsmath}

\newcommand{\auskommentieren}[1]{}

\newcommand{\beq}{\begin{equation}}
\newcommand{\eeq}{\end{equation}}
\DeclareMathOperator{\graph}{graph}
\DeclareMathOperator{\diam}{diam}
\DeclareMathOperator{\dist}{dist\ }

\newtheorem{remark}[theorem]{Remark}

\title{$L^{\infty}$-error estimate for the finite element method on two dimensional surfaces} 

\author{Heiko Kr\"oner\thanks{Bereich Optimierung und Approximation, Fachbereich Mathematik, Bundesstrasse 55, 20146 Hamburg, Germany,
{\tt Heiko.Kroener@uni-hamburg.de}}}

\begin{document}
\maketitle
\slugger{mms}{xxxx}{xx}{x}{x--x}

\begin{abstract}
We approximate the solution of the equation
\begin{equation}
-\Delta_S u+u = f
\end{equation}
on a two-dimensional, embedded, orientable, closed surface $S$ 
where $-\Delta_S$ denotes the Laplace Beltrami operator on $S$ 
by using continuous, piecewise linear finite elements on a triangulation of $S$
with flat triangles.
We show that the $L^{\infty}$-error is of order $O(h^2|\log h|)$ as in the corresponding
situation in an Euclidean setting.
\end{abstract}

\section{Introduction}
During the last years 
several articles appeared which deal with the numerical solution of linear partial
differential equations which are defined on a hypersurface in $\mathbb{R}^3$. 
Roughly spoken their common aim is to show that concepts and properties which are well-known
in an Euclidean setting carry over to the surface case.
Without claiming completeness we summarize some steps towards this goal.

In \cite{Dziuk1988} the finite element approximation of the Laplace-Beltrami equation
\begin{equation}
-\Delta_S u = f
\end{equation}
on a surface $S$ with continuous, piecewise linear elements  (on a polyhedral approximation of $S$) is presented and it is shown that the $L^2$- and $H^1$- error estimates known from the corresponding Euclidean setting carry over to this case. In \cite{DziukElliott2007} this idea is extended to a semi-discrete approximation of linear parabolic equations which are defined on a (moving) hypersurface for which the motion is a priori given by a smooth one parameter family of diffeomorphisms of a fixed initial surface. Here, one has to take care of the fact that the time derivative is defined suitably which is tackled by the concept of the material derivative and the
spatial discretization uses a moving mesh (the method is called ESFEM). Furthermore, $L^{\infty}(L^2)$- and $L^2(H^1)$-error estimates 
are shown and in \cite{DziukElliott2013} the $L^{\infty}(L^2)$-estimate is improved to the optimal order of  $O(h^2)$.
We refer to \cite{DziukElliott2013Acta} for an survey of finite element methods for surface
PDEs.

We mention further contributions to this topic in the literature. 
In \cite{DziukLubichMansour2012} Runge-Kutta methods
known from the semi-linear Euclidean setting, cf. \cite{LubichOstermann1993, LubichOstermann1995a, 
LubichOstermann1995b}, are adapted to 
ESFEM to obtain a fully discrete approximation of the linear parabolic equation in 
combination with the moving surface. See also \cite{LubichMansourVenkataraman2013} for  a backwards 
difference time discretization of this problem.
In \cite{ElliottStyles2012} an additional tangential motion of the grid for ESFEM is introduced to 
improve the mesh quality, or more precisely, to compensate a motion related
possibly
deterioration of the mesh. 
In \cite{OlshanskiiReuskenGrande2009} finite element spaces that are induced by 
triangulations of an 'outer' domain are used to discretize partial differential equations 
on a surface, see also \cite{OlshanskiiReuskenGrande2010}.  
In \cite{OlshanskiiReuskenXu2014, OlshanskiiReusken2014} an Eulerian finite element method for solving linear 
parabolic partial differential equations is presented and  
a stabilized finite element method for linear parabolic equations on surfaces is studied.
In \cite{DziukElliottHeine2010} a $h$-narrow band finite element method for linear elliptic equations on implicit surfaces is studied. See also \cite{BertalmioChengOsherSapiro2001} for variational problems and partial differential equations on implicit surfaces.
In \cite{DednerMadhavanStinner2013} an analysis of the discontinuous Galerkin method for linear elliptic
problems on surfaces is carried out. 

In our paper we show that the well-known $L^{\infty}$-estimate for the finite element 
approximation of linear elliptic equations in a  two dimensional Euclidean setting, 
cf. \cite{Scott1976} and compare also \cite{Natterer1975, Nitsche1975, Haverkamp1984}, 
carries over to the case of a linear elliptic equation on a surface in $\mathbb{R}^3$
which seems to be omitted in the literature until now according to the author's knowledge.

We refer to \cite{BaumgardnerFrederickson1985} where an icosahedral discretization of the two-sphere is used to solve the Laplace-Beltrami equation on the two-sphere. 
There it is claimed (without detailed justification) that the quadratic order of the 
$L^{\infty}$-interpolation error immediately carries over to an 
$L^{\infty}$-estimate of quadratic order 
for the discretization error, see the passage following Table 2 on page 1114 in \cite{BaumgardnerFrederickson1985}, which is wrong, of course. 

Our paper is organized as follows. In the remaining part of the introduction
we present the general setting and formulate our partial differential equation. In Section \ref{elementary} we introduce our notation, state some basic facts which will be used in the sequel several times and present the discretization of the equation. In
Section \ref{17} we present for completeness 
in the surface case 
 the known $H^1$- 
and $L^2$-error estimates, cf. \cite{Dziuk1988}. 
In Section \ref{32} we state in Theorem \ref{11} our main result about the $L^{\infty}$-error estimate in the surface case and present a proof.

We sketch the idea of the proof.
We prove the estimate
pointwise by using an approximative Green's function $\tilde v$ on the surface.
The latter function is obtained by lifting a cutted-off Euclidean Green's function
from the tangent plane to the surface at which the appearing discrepancy to an exact Green's 
function on the surface is -- in case of relevance -- supressed by the $L^2$-error estimates 
from Section \ref{17}. 
We define a finite element approximation of $\tilde v$ for which we prove an error 
estimate in the $W^{1,1}$-norm which has the same order as in the corresponding
Euclidean case. In doing so we adapt the argumentation from \cite{Scott1976}.

Let $S$ be a smooth two-dimensional, embedded, orientable, closed hypersurface in $\mathbb{R}^3$. We triangulate the surface by a family $T_h$ of flat triangles with corners (i.e. nodes) lying on $S$. We denote the surface of class $C^{0,1}$ given by the union of the triangles $\tau \in T_h$ by $S_h$; the union of the corresponding nodes is denoted by $N_h$. Here, $h>0$ denotes a discretization parameter which is related to the triangulation in the following way.
For $\tau \in T$ we define the diameter $\rho(\tau)$  of the smallest disc containing $\tau$, the diameter
 $\sigma(\tau)$ of the largest disc contained in $\tau$ and
\begin{equation}
h = \max_{\tau \in T_h}\rho(\tau), \quad \gamma_h = \min_{\tau \in T_h}\frac{\sigma(\tau)}{h}.
\end{equation}
We assume that the family $(T_h)_{h>0}$ is quasi-uniform, i.e. $\gamma_h \ge \gamma_0 >0$.
We let 
\begin{equation}
V_h = \{v\in C^0(S_h): v_{|\tau}\ \text{linear for all}\ \tau \in T_h \}
\end{equation}
be the space of continuous piecewise linear finite elements.

We assume $f\in L^2(S)$  and our goal is to prove error estimates for a finite element approximation of the unique solution $u\in H^2(S)$ of the PDE 
\begin{equation} \label{62}
-\Delta_S u + u = f
\end{equation}
where $\Delta_S$ is the Laplace-Beltrami operator on $S$. In Section \ref{32} we will assume that $f$ is in addition so that $ u \in W^{2,\infty}(S)$.

\begin{remark}
After submitting a first version of the present article to arXiv the author became aware of the fact that Theorem \ref{11} has been proved in \cite{demlow}.
\end{remark}

\section{Notations, elementary observations and discrete formulation} \label{elementary}
Let $N$ be a tubular neighborhood of $S$ in which 
the Euclidean metric of $N$ can be written in the coordinates $(x^0, x)=(x^0, x^i)$ of the tubular neighborhood as
\begin{equation}
\bar g_{\alpha \beta} = (dx^0)^2 + \sigma_{ij}(x)dx^idx^j.
\end{equation} 
Here, $x^0$ denotes the globally (in $N$) defined signed distance to $S$ and 
$x=(x^i)_{i=1,2}$ local coordinates for $S$.

For small $h$ we can write $S_h$ as graph (with respect to the coordinates
of the tubular neighborhood) over $S$, i.e.
\begin{equation} \label{30}
S_h = \graph \psi = \{(x^0, x): x^0 = \psi(x), x \in S\}
\end{equation}
where $\psi=\psi_h \in C^{0,1}(S)$ suitable. Note, that 
\begin{equation} \label{31}
|D\psi|_{\sigma}\le c h, \quad |\psi|\le ch^2.
\end{equation}
The induced metric of $S_h$ is given by
\begin{equation}
g_{ij}(\psi(x), x) = \frac{\partial \psi}{\partial x^i}(x) \frac{\partial \psi}{\partial x^j}(x) + \sigma_{ij}(x).
\end{equation}
Hence we have for the metrics, their inverses and their determinants
\begin{equation}
g_{ij}=\sigma_{ij}+O(h^2), \quad 
g^{ij} = \sigma^{ij}+O(h^2) \quad \text{and} \quad
g = \sigma + O(h^2)|\sigma_{ij}\sigma^{ij}|^{\frac{1}{2}}
\end{equation}
where we use summation convention.

Let $f \in W^{1,p}(S)$, $g \in W^{1,p^{*}}(S)$, $1\le p \le \infty$ and $p^{*}$
H\"older conjugate of $p$. We define the so-called lift 
$\tilde f$ of $f$ to $S_h$ by $f(x) = \tilde f(\psi(x), x)$, $x\in S$, and correspondingly 
for $g$ (more generally, we can do this procedure whenever we have two graphs in the same coordinate system and denote it by the terminus lift, furthermore, this terminus can be obviously extended to subsets). 
In local coordinates $x=(x^i)$ of $S$ hold
\begin{equation} \label{4}
\int_S \left<D f, D g\right> = \int_S \frac{\partial f}{\partial x^i}\frac{\partial g}{\partial x^j}\sigma^{ij}(x)\sqrt{\sigma(x)}dx^idx^j,
\end{equation}
\begin{equation} \label{5}
\int_{S_h} \left<D \tilde f, D \tilde  g\right> = \int_{S} \frac{\partial f}
{\partial x^i}\frac{\partial g}{\partial x^j}g^{ij}(\psi(x), x)\sqrt{g(\psi(x), x)}
dx^idx^j,
\end{equation}
\begin{equation}  \label{101}
\int_S \left<D f, D g\right> = \int_{S_h} \left<D \tilde f, D \tilde  g\right> + 
O(h^2)\|f\|_{W^{1,p}(S)}
\|g\|_{W^{1,p^{*}}(S)},
\end{equation}
and similarly, 
\begin{equation} \label{100}
\int_S f =   \int_{S_h}\tilde f+ O(h^2)\|f\|_{L^1(S)}
\end{equation}
where now $f\in L^1(S)$ is sufficient.

The bracket $\left<u,v\right>$ denotes here the scalar product of two tangent vectors $u,v$ (or their covariant counterparts), and later also the application of a distribution $u$ to a test function $v$; which meaning is on hand will be clear from context. $\|\cdot \|_{W^{k,p}}$ denotes the usual Sobolev norm, $|\cdot |_{W^{k,p}}=\sum_{|\alpha|=k}\|D^{\alpha}\cdot \|_{L^p}$ and $H^k=W^{k,2}$.

\begin{remark}\label{70}  
Let $z_0\in S$ and $T_{z_0}S$ be the tangent plane of $S$ in $z_0$. A ball 
$B_{2r_1}(z_0)\subset S$, $r_1>0$ suitable, and a corresponding portion of $S_h$ can be written
as a graph over a corresponding subset 
$U$ of $T_{z_0}S$ in 
a (perpendicular) Euclidean coordinate system $(x^0, x^1, x^2)$. 
Here, $x^1, x^2$
denote Euclidean coordinates in $T_{z_0}S$ with center in $z_0$ and $x^0$ denotes the coordinate
axis perpendicular to $T_{z_0}S$ so that $T_{z_0}S = \{x^0=0\}$.
One can consider lifts of functions between these three (pieces of) surfaces with
 respect to this Euclidean representation as graph.
Analogous estimates to (\ref{100}), (\ref{101}) hold except for lifts from or to $U$, for these $O(h^2)$ has to be replaced by $O(\max(\diam (\supp f), \diam (\supp g))^2)$. 
 Since $S$ is compact all constants in the estimates and $r_1$
can be chosen independently from $z_0$.
\end{remark}

Since the properties and aspects needed to prove a priori error estimates for finite element 
approximations are formulated in terms of integrals 
these observations concerning the transformation behavior of integrals  
essentially imply that the known error estimates from the Euclidean setting 
carry over to the surface case as far as convergence 
of at most quadratic order is concerned.

We define 
\begin{equation} \label{8}
a:W^{1,p}(S)\times W^{1,p^{*}}(S)\rightarrow \mathbb{R}, \quad a(u,v) =\int_S \left<Du, Dv\right>+uv 
\end{equation}
and
\begin{equation} \label{9}
a_h:W^{1,p}(S_h)\times W^{1,p^{*}}(S_h)\rightarrow \mathbb{R}, \quad a(u_h,v_h) =\int_{S_h} \left<Du_h, Dv_h\right>
+u_h v_h. 
\end{equation}
The finite element approximation of $u$ in (\ref{62}) is defined as 
the unique $u_h\in V_h$ with
\begin{equation} \label{3}
a_h(u_h, \varphi_h) = \int_{S_h}f_h\varphi_h \quad \forall \varphi_h \in V_h
\end{equation}
where $f_h$ is the lift of $f$ to $S_h$. Note, that existence follows from uniqueness.

Constants that do not depend on $h$ or $z_0$ are denoted by $c$ or $c_0, c_1$ if several appear and should
be specifiable.

Functions labelled by capital letters denote (without explicit declaration) the lift via the representation as graph with respect to the tubular neighborhood $N$ of $S$ of the function
denoted with the corresponding small letter to the other surface, e.g. assume $w$ is defined on $S$ then $W$ denotes the lift of $w$ to $S_h$ and vice versa, i.e. if $w$ is defined on $S_h$ then $W$ denotes the lift of $w$ to $S$.

 \section{The $H^1$-estimate and $L^2$-estimate} \label{17}
For completeness we give in this Section a proof of the well-known estimates stated in Lemma \ref{63} and
Lemma \ref{64}, cf. \cite{Dziuk1988}.

\begin{lemma} \label{63}
We have
\begin{equation}Ê
 \|u-U_h\|_{H^1(S)} \le O(h)\|f\|_{L^2(S)}.
\end{equation}
\end{lemma}
\begin{proof}
Let $\varphi_h \in V_h$ arbitrary
then
\begin{equation} \label{1}
\begin{aligned}
\|u-U_h\|^2_{H^1(S)} = &a(u-U_h, u-U_h) \\
=& a(u-U_h, u-\Phi_h)+a(u-U_h, \Phi_h-U_h)\\
\le&  \|u-U_h\|_{H^1(S)}\|u-\Phi_h\|_{H^1(S)}
+a(u-U_h, \Phi_h-U_h).
\end{aligned}
\end{equation}
We rewrite
\begin{equation}
 \begin{aligned}
a(u-U_h, \Phi_h-U_h) =& \int_S f(\Phi_h-U_h)-\int_{S_h}f_h(\varphi_h-u_h)\\
& +O(h^2)\|U_h\|_{H^1(S)}
\|\Phi_h-U_h\|_{H^1(S)} \\
=&  O(h^2)(\|U_h\|_{H^1(S)}+\|f\|_{L^2})
\|\Phi_h-U_h\|_{H^1(S)} 
 \end{aligned}
\end{equation}
so that we obtain from (\ref{1})
\begin{equation}
\begin{aligned}
\| u-& U_h\|_{H^1(S)} 
\le 2\max(m_1, m_2)
\end{aligned}
\end{equation}
where
\begin{equation}
m_1 = \inf_{\varphi_h \in V_h} \|u- \Phi_h\|_{H^1(S)}
\end{equation}
and 
\begin{equation}
 m_2=\left\{O(h^2)\|f\|_{L^2(S)}
(\inf_{\varphi_h \in V_h} \|\Phi_h-u\|_{H^1(S)}+\|u-U_h\|_{H^1(S)})\right\}^{\frac{1}{2}}.
\end{equation}
Let $\tilde u$ be the lift of $u$ to $S_h$ then
\begin{equation}
\begin{aligned}
\inf_{\varphi_h \in V_h}\|u- \Phi_h\|_{H^1(S)} \le& \|u-\tilde u\|_{H^1(S)}
+\inf_{\varphi_h \in V_h}\|\tilde u- \varphi_h\|_{H^1(S_h)}+O(h^2)\|f\|_{L^2(S)}\\
\le& O(h)\|f\|_{L^2(S)}.
\end{aligned}
\end{equation}
Putting these estimates together yields the claim.
\end{proof}

The estimate in the $L^2$-norm can be improved.
\begin{lemma} \label{64}
We have
\begin{equation}
 \|u-U_h\|_{L^2(S)} \le O(h^2)\|f\|_{L^2(S)}.
\end{equation}
\end{lemma}
\begin{proof}
Let $w \in H^2(S)$ be the unique solution of $-\Delta_S w+w = u-U_h$ and $w_h \in V_h$ the 
corresponding unique finite element solution to the right-hand side $\tilde u-u_h$. 
Then we have
\begin{equation}
\begin{aligned}
\int_S(u-U_h)^2 =& a(u-U_h, w) \\
=& a(u-U_h, w-W_h)+a(u-U_h, W_h) \\
\le&\|u-U_h\|_{H^1(S)}\|w-W_h\|_{H^1(S)}\\
& +O(h^2)(\|U_h\|_{H^1(S)}
+\|f\|_{L^2(S)})\|W_h\|_{H^1(S)} \\
\le& ch\|u-U_h\|_{H^1(S)}\|u-U_h\|_{L^2(S)} \\
&+ O(h^2)(\|U_h\|_{H^1(S)}+\|f\|_{L^2(S)})\|u-U_h\|_{L^2(S)}.
\end{aligned}
\end{equation}
\end{proof}

\section{The $L^{\infty}$-estimate} \label{32}

We assume that $f\in L^2(S)$ is in addition so that $ u \in W^{2,\infty}(S)$.
The following theorem states our main result.
\begin{theorem} \label{11}
There holds
\begin{equation}
\|u-U_h\|_{L^{\infty}(S)}\le c h^2|\log h|\|u\|_{W^{2,\infty}(S)}.
\end{equation}
\end{theorem}
The proof of the corresponding Euclidean statement is well-known, cf. \cite{Scott1976}.

The purpose of the remaining part of this section is to prove Theorem \ref{11}.

Let $z_0 \in S$ and 
\begin{equation}
\varphi_S : U \rightarrow B_{2r_1}(z_0), \quad (0,x) \mapsto (\psi_S(x),x), \quad x=(x^1, x^2)
\end{equation}
be the representation as graph of $B_{2r_1}(z_0)\subset S$ over $U\subset T_{z_0}S$
according to Remark \ref{70}.

\begin{definition} \label{92}
We set $B_j= B_{jr_1}(z_0)$, $j=1,2$,
 $\varphi = \varphi_S$ and  let $v$ be the Euclidean
 Green's function with respect to $-\Delta +I$ in $T_{z_0}S\equiv \mathbb{R}^2$ with singularity in $z_0$,
 i.e., more precisely,
 \begin{equation} \label{86}
  \begin{aligned}
   -\Delta v + v =& \delta_{z_0} \quad \text{in } B_{100}(z_0)\subset \mathbb{R}^2\\
   \partial_n v=& 0 \quad \text{on } \partial B_{100}(z_0).
  \end{aligned}
 \end{equation}
 Let $\zeta\in C^{\infty}_0(B_{\frac{3}{2}r_1}(z_0))$, $\zeta\equiv 1$ in $B_1$, be a cut-off function
 and set 
 \begin{equation}
 \tilde v(x) = v(\varphi^{-1}(x))\zeta(x), \quad 
  \tilde l(x) = l(\varphi^{-1}(x))\zeta(x), \quad x \in B_2,
 \end{equation} 
 where $l(z)=\log|z-z_0|$.
 \end{definition}
 
 There holds
 \begin{equation} \label{20}
  v-\frac{1}{2\pi}l \in H^2(B_{100}(z_0)),
 \end{equation}
 cf. \cite[Lemma 1, page 687]{Scott1976}.
W.l.o.g. we may assume $\|\tilde v\|_{C^2(B_2\setminus
 B_1)} \le c_0$ and
 \begin{equation} \label{71}
\|\tilde v-\frac{1}{2\pi}\tilde l\|_{W^{2,2}(S)} \le c
 \end{equation} 
 where $c_0, c$  are independent from $z_0$.
 
 The next Lemma shows in which sense $\tilde v$ is an approximative Green's function.
 \begin{lemma} \label{90}
  Let $\tilde v$ be as in Definition \ref{92}. Let $w \in H^1(S)\cap C^0(S)$ then
  \begin{equation}
   |\left<-\Delta_S\tilde v+\tilde v, w\right>
  - w(z_0)|\le  c \|w\|_{L^2(S)}.
   \end{equation}
 \end{lemma}
\begin{proof}
 Let $\eta\in C^{\infty}_0(B_{r_1}(z_0))$, 
 $\eta_{|B_{r_1/2}(z_0)}\equiv 1$, $|\eta|\le c$ and $|D\eta|\le c$.  We write
 \begin{equation}
 w =  \eta w + (1- \eta)w = w^1+w^2 
 \end{equation}
 and have (in local coordinates induced from $\varphi$) 
 \begin{equation} \label{7}
 \begin{aligned}
 \left<-\Delta_S \tilde v + \tilde v , w\right> =&  \left<-\Delta_S \tilde v + \tilde v , w^1\right>+\left<-\Delta_S \tilde v + \tilde v , w^2\right> \\
 =& \int_U \left(g^{ij}\frac{\partial  \tilde v}{\partial x^i}\frac{\partial w^1}{\partial x^j}+
  \tilde v w^1\right)\sqrt{g}+
 \int_U \left(g^{ij}\frac{\partial \tilde v}{\partial x^i}\frac{\partial w^2}{\partial x^j}+ \tilde vw^2\right)\sqrt{g}. 
  \end{aligned}
 \end{equation}
 We rewrite the second integral on the right-hand side of (\ref{7}) as
 \begin{equation} \label{6}
 \begin{aligned}
 \int & ( \tilde v-\delta^{ij}
 \frac{\partial^2 \tilde v}{\partial x^i\partial x^j})\sqrt{g}w^2-\int \frac{\partial}{\partial x^j}g^{ij}
 \frac{\partial  \tilde v}{\partial x^i}\sqrt{g}w^2\\
 &- \int
 g^{ij}\frac{\partial  \tilde v}{\partial x^i}\frac{1}{2}\sqrt{g}g^{kl}
 \frac{\partial}{\partial x^j}g_{kl}w^2
 +\int (\delta^{ij}-g^{ij})
 \frac{\partial^2 \tilde v}{\partial x^i\partial x^j}\sqrt{g}w^2 \\
 =& O(1)\|w\|_{L^2(S)}
 \end{aligned}
 \end{equation}
  where we used integration by parts and all integrals are over $U\setminus \varphi^{-1}(B_{\frac{r_1}{2}}(z_0))$.

Let $r$ denote the distance to $z_0$ in $T_{z_0}S$ then we obtain 
 \begin{equation} \label{80}
  \begin{aligned} 
   |\delta^{ij}-g^{ij}| \le& c r^2, \quad
   |g^{ij}| \le c, \quad
   \left|\frac{\partial}{\partial x^k}g_{ij}\right| 
   \le cr, \\
   & \left | \frac { \partial  \tilde l} {\partial x^i} \right| \le \frac{c}{r}
   , \quad \left| \frac{\partial^2  \tilde l }{\partial x^i\partial x^j}\right|\le \frac{c}{r^2}.
   \end{aligned}
   \end{equation}
   
 We rewrite the first integral on the right-hand side of (\ref{7}) as
 \begin{equation}
 \begin{aligned}
 \int_U &\left(g^{ij}\frac{\partial \tilde  v}{\partial x^i}\frac{\partial w^1}
 {\partial x^j}+ \tilde v w^1\right)\sqrt{g}\\
 =& \int_U \left(g^{ij}\frac{\partial \tilde  v}{\partial x^i}\frac{\partial w^1}
 {\partial x^j}+ \tilde v w^1\right)(\sqrt{g}-1) + \int_U (g^{ij}-\delta^{ij}) \frac{\partial \tilde  v}{\partial x^i}\frac{\partial w^1}
 {\partial x^j}\\
 &+ \int_U \left(\delta^{ij}\frac{\partial \tilde  v}{\partial x^i}\frac{\partial w^1}
 {\partial x^j}+ \tilde v w^1\right) \\
 =& w(z_0)+O(1)\|w\|_{L^2(S)}
 \end{aligned}
 \end{equation}
where we used (\ref{80}), H\"older's inequality, (\ref{86}) and that we are allowed to perform integration by parts in the integrals with the factors $(\sqrt{g}-1)$ and $g^{ij}-\delta^{ij}$, note, that
   \begin{equation} \label{72}
   \begin{aligned}
    \int_{U}&\left|(\delta^{ij}-g^{ij})
 \frac{\partial^2 \tilde v}{\partial x^i\partial x^j} w^1\right| \\
 \le&   \int_{U}\left|(\delta^{ij}-g^{ij})\left(
 |\frac{\partial^2 (\tilde v-\tilde l)}{\partial x^i\partial x^j}|+|
 \frac{\partial^2 \tilde l}{\partial x^i\partial x^j}|\right) w^1\right| \\
  \le& c\|w\|_{L^2(S)}.
  \end{aligned}
   \end{equation}
 \end{proof}
 
\begin{remark} 
From now, we denote the approximative Green's function $\tilde v$ by $g$ (there will be no ambiguity with the symbol for the determinant of the metric).  
\end{remark}
 
We define an approximation $g_h\in V_h$ of $g$ by
\begin{equation} \label{110}
a_h(g_h, v_h)=a(g, V_h)
\end{equation}
for all $v_h \in V_h$.

\begin{lemma}
Assume 
\begin{equation} \label{13}
\|g-G_h\|_{W^{1,1}(S)} \le c h |\log h|
\end{equation}
then Theorem \ref{11} follows.
\end{lemma}
\begin{proof}
From Lemma \ref{90} we conclude 
\begin{equation} \label{12}
\begin{aligned}
(u-U_h)(z_0) =& a(g, u-U_h) +O(h^2)\|f\|_{L^2(S)}\\
=& a_h(G-g_h, U-u_h)+O(h^2)\|g\|_{W^{1,1}(S)}\|u-U_h\|_{W^{1, \infty}(S)}\\
& +a_h(g_h, U-u_h).
\end{aligned}
\end{equation}
We estimate
\begin{equation}
 \begin{aligned}
  a_h(g_h, U-u_h)=& a(G_h, u)-a_h(g_h, u_h)+O(h^2) \|G_h\|_{W^{1,1}(S)} \|u\|_{W^{1,\infty}(S)}\\
=& O(h^2)\left( \|G_h\|_{W^{1,1}(S)} \|u\|_{W^{1,\infty}(S)}+\|G_h\|_{L^1(S)}\|u\|_{W^{2,\infty}(S)}\right).
 \end{aligned}
\end{equation}
Furthermore, we have
\begin{equation} 
\begin{aligned}
a_h(G-g_h, U-u_h)
=& a_h(G-g_h, U-v_h)+a_h(G-g_h, v_h-u_h), 
\end{aligned}
\end{equation}
rewrite the second summand by using Lemma \ref{90} as
\begin{equation}
\begin{aligned}
a_h(G-&g_h, v_h-u_h)  \\
=&  a(g, V_h-U_h) + O(h^2)\|g\|_{W^{1,1}(S)}\|V_h-U_h\|_{W^{1, \infty}(S)}\\
& -a_h(g_h, v_h-u_h)\\
\le& O(h^2)\|g\|_{W^{1,1}(S)}\|V_h-U_h\|_{W^{1, \infty}(S)} 
\end{aligned}
\end{equation}
and estimate the first summand as follows
\begin{equation}
\begin{aligned}
|a_h(G-g_h, U-v_h)| \le& \|g-G_h\|_{W^{1,1}(S)}\|u-V_h\|_{W^{1,\infty}(S)}
\end{aligned}
\end{equation}
We let $v_h\in V_h$ be the interpolation of $u$ and obtain the claim from 
\begin{equation}
\begin{aligned}
\|u-U_h\|_{W^{1,\infty}(S)} \le& \|u-V_h\|_{W^{1,\infty}(S)}+\|V_h-U_h\|_{W^{1,\infty}(S)} \\
\le& \|u-v_h\|_{W^{1,\infty}(S)}\\
&+ch^{-1}\left(\|V_h-u\|_{W^{1,2}(S)}
+\|u-U_h\|_{W^{1,2}(S)}\right)
\end{aligned}
\end{equation}
which holds in view of estimate (\ref{16}) .
\end{proof}

In order to show  (\ref{13}) we prove as first step the following Lemma.
\begin{lemma} \label{35}
We have
\begin{equation}
\|g-G_h\|_{L ^2(S)}\le ch
\end{equation}
where $c$ is independent of $z_0$.
\end{lemma}
\begin{proof}
Let $\tau$ be a triangle in $T_h$ containing the lift $\tilde z_0$ of $z_0$, and let $q$ be the linear function with
\begin{equation}
\int_{\tau}qp = p(z_0)
\end{equation}
for all linear functions $p$. Because $\tau$ contains a disk of radius $\gamma_0 h$ we see that
\begin{equation} \label{14}
\sup_{\tau}|q| \le c h^{-2}.
\end{equation}
We extend the domain of definition of $q$ to $S_h$ by zero and set $\tilde \delta = Q$.

We
define 
\begin{equation}
\begin{aligned}
\psi(v) =& a(g,v)-\left<\delta_{z_0}, v\right>, \quad v \in W^{1, \infty}(S)\cap C^0(S),
\end{aligned}
\end{equation}
From Lemma \ref{90} we deduce that
\begin{equation}
|\psi(v)| \le c_0 \|v\|_{L^2(S)}
\end{equation}
so that by Hahn-Banach Theorem $\psi$ can be extended to a linear functional on $L^2(S)$-- denoted by $\psi$ as well -- with norm $\le c_0$. W.l.o.g we may assume that $\psi \in L^2(S)$. 

Let $w\in H^1(S)$ and $w_h\in V_h$ be the solutions of 
\begin{equation}
\begin{aligned}
a(w,v) =& \int_S\psi v \quad \forall v \in H^1(S) \\
a_h(w_h, v_h) = &\int_{S_h}\Psi v_h \quad \forall v_h \in V_h. 
\end{aligned}
\end{equation}
 Lemma \ref{63}  leads to
\begin{equation}
\|w-W_h\|_{H^1(S)} \le ch
\end{equation} 
where $c$ independent from $z_0$. 

Let $z_h \in V_h$ with 
\begin{equation}
a_h(z_h, v_h) = \left<\delta_{\tilde z_0}, v_h\right> = v_h(\tilde z_0).
\end{equation}
Let $\tilde g$ solve
\begin{equation}
-\Delta_S \tilde g + \tilde g = \tilde \delta.
\end{equation}
Since $z_h$ can be seen as finite element approximation of $\tilde g$ we have
in view of Lemma \ref{63}
\begin{equation}
\|\tilde g-Z_h\|_{H^s(S)} \le c h^{2-s}\|\tilde \delta\|_{L^2(S)}\le c h^{1-s}, \quad s=0,1.
\end{equation}
In view of 
\begin{equation}
\begin{aligned}
a_h(g_h-w_h, v_h) =& a(g, V_h)-\int_{S_h}\Psi v_h \\
=& a(g, V_h)-\int_S \psi V_h+O(h^2)\|\psi\|_{L^2(S)}\|V_h\|_{L^2(S)} \\
=& \left<\delta_{z_0}, V_h\right> +O(h^2)\|\psi\|_{L^2(S)}\|V_h\|_{L^2(S)}, \quad \forall v_h\in V_h
\end{aligned}
\end{equation}
we deduce 
\begin{equation}
\|z_h-(g_h-w_h)\|_{L^2(S)} \le O(h^2).
\end{equation}
To estimate 
\begin{equation}
\|g-G_h\|_{L^2(S)} = \|(g-w-\tilde g)+(\tilde g-Z_h)+(w-W_h)\|_{L^2(S)}+O(h^2)
\end{equation}
we need to estimate $\|g-w-\tilde g\|_{L^2(S)}$. Let $\varphi \in C^{\infty}(S)$ and $\tilde w$ a solution of
\begin{equation}
-\Delta_S \tilde w + \tilde w = \varphi
\end{equation}
then
\begin{equation}
\begin{aligned}
\int_S(g-w-\tilde g)\varphi =& a(g-w-\tilde g, \tilde w) \\
=& \left<\delta_{z_0}-\tilde \delta, \tilde w\right> \\
=& \left<\delta_{z_0}-\tilde \delta, \tilde w-\tilde w_I\right> \\
\le& \|\tilde w-\tilde w_I\|_{L^{\infty}(S)} + \|\tilde \delta\|_{L^2(S)}\|\tilde w-\tilde w_I\|_{L^2(S)} \\
\le& O(h)\|\varphi\|_{L^2(S)}.
\end{aligned}
\end{equation}
Here, $\tilde W_I$ denotes the linear interpolation of $\tilde w$, $\tilde w_I$ its lift to $S$ and we used, cf. \cite[Theorem 4.4.20]{BrennerScott},
\begin{equation} \label{16}
h^{\frac{2}{p}}\|\chi-\chi_I\|_{L^{\infty}(S)}+\sum_{j=0}^1h^j\|\chi-\chi_I\|_{W^{j,p}(S)} \le c h^2 \|\chi\|_{W^{2,p}(S)}, \quad 1 \le p\le \infty,
\end{equation}
for $\chi \in H^2(S)$ and $\chi_I$ the linear interpolation of $\chi$ (, and the right-hand side possibly unbounded).
\end{proof}

\begin{remark} \label{44} 
(i) Estimate (\ref{13}) follows immediately if we show
\begin{equation} \label{42}
 \|\tilde l-\tilde L_h\|_{W^{1,1}(S)} \le c h |\log h|
 \end{equation}
 where $\tilde l_h\in V_h$ is defined by 
 \begin{equation}
 a_h(\tilde l_h, v_h) = a(\tilde l, V_h) \quad \forall v_h \in V_h.
 \end{equation}
 (ii) There holds
 \begin{equation} \label{39}
  \|\tilde l-\tilde L_h\|_{L^2(S)} \le ch.
 \end{equation}
\end{remark}
\begin{proof}
(i) 
We have
\begin{equation}
a_h(g_h-\frac{1}{2\pi}\tilde l_h, v_h) = a(g-\frac{1}{2\pi}\tilde l, V_h)
\end{equation}
so that in view of (\ref{20}) we conclude from Lemma \ref{63} that
\begin{equation} \label{89}
\|G_h-\frac{1}{2\pi}\tilde L_h-(g-\frac{1}{2\pi}\tilde l)\|_{H^1(S)} \le c h
\end{equation}
and the triangle inequality implies (\ref{13}). 

(ii) Use (\ref{89}), the triangle inequality and  Lemma \ref{35}.
\end{proof}

In the remaining part of this section we prove (\ref{42}).
We recall that $l(z)=\log |z-z_0|$ is defined in $T_{z_0}S\equiv \mathbb{R}^2$, that $r$ denotes the distance to $z_0$ in $T_{z_0}S$ and state that $l$ has bounded mean oscillation in the following sense.
\begin{lemma} \label{37}
 Let $z_1\in \mathbb{R}^2$ and $0<\rho<\infty$. Then there is a constant $l_0\in \mathbb{R}$ 
 depending on $z_1$ and $\rho$ such that
 \begin{equation}
  \int_{\{|z-z_1|\le \rho\}}(l-l_0)^2 \le 9 \pi \rho^2.
 \end{equation}
 \end{lemma}
 \begin{proof}
  This is the assertion of \cite[Lemma 2 on page 688]{Scott1976}. 
 \end{proof}
 \begin{remark} \label{36}  
 In the following we will consider lifts of objects defined on $B_{2r_1}(z_0)\subset S$, $U\subset T_{z_0}S$ or a suitable portion of $S_h$ to another of these three surfaces with respect to the representation as graph over $U$ in (perpendicular) Euclidean coordinates as described in Remark \ref{70}. By 
adding the superscripts $S$, $T$ or $h$ we indicate to which surface the object is lifted, e.g. 
let $M \subset B_{2r_1}(z_0)\subset S$ then $M^T$ denotes its lift to $T_{z_0}S$.
Similar correction terms as in (\ref{101}) and (\ref{100}) appear when we lift integrands of (with a power of $r$) weighted $W^{1,p}$-norms, i.e. if we estimate such a norm then the lift produces (at most) a constant as factor 
on the right-hand side of the estimates. 
\end{remark}
 
We estimate the error $E=\tilde l-\tilde L_h$ near $z_0$.
\begin{lemma} \label{40}
Let $0<\rho <c_1h$ be given and $B=\{|z-z_0|\le \rho\}$ a ball in $T_{z_0}S$. Then
\begin{equation}
\int_{S\cap B^S}|\varphi^{-1}_{z_0}(z)-z_0|^{\beta}|DE|^p 
\le c \rho^{\beta}h^{2-p}
\end{equation}
for $1\le p < \beta +2$.
\end{lemma}
\begin{proof}
 We have $|D\tilde l| \le \frac{c}{r}$ where $r=|\varphi^{-1}_{z_0}(z)-z_0|$. In view of 
 Remark \ref{36} we may  w.l.o.g. consider $r$ as a function as well on $B^S$ and $B=B^T$ and get
 \begin{equation}
  \begin{aligned}   
 \int_{S\cap B^S}|\varphi_{z_0}(z)-z_0|^{\beta}|DE|^p 
 \le& c \int_{S\cap B^S}r^{\beta}(|D \tilde l|^p+|D\tilde L_h|^p) \\
 \le& c \rho^{\beta+2-p} + \int_{B^h}r^{\beta}|D\tilde l_h|^p.
  \end{aligned}
 \end{equation}
By Lemma \ref{37} there is $l_0 \in \mathbb{R}$ so that
\begin{equation} \label{38}
 \|l-l_0\|_{L^2(B^T)} \le c h.
\end{equation}
We get (using an inverse estimate to bound a $W^{1, \infty}$- by a $L^2$-norm)
\begin{equation}
\begin{aligned}
 \int_{B^h}r^{\beta}|D\tilde l_h|^p=& \int_{B^h}r^{\beta}|D(\tilde l_h-l_0)|^p \\
 \le& c\rho^{\beta} h^2\sup_{B^h}|D(\tilde l_h-l_0)|^p \\
 \le& c \rho^{\beta}h^{2}h^{-2p}\|\tilde l_h-l_0\|^p_{L^2(B^h)}\\
  \le& c \rho^{\beta}h^{2}h^{-2p}(\|\tilde L_h-\tilde l\|_{L^2(B^S)}+\| l-l_0\|_{L^2(B^T)})^p \\
  \le& c \rho^{\beta}h^{2-p}
\end{aligned}
 \end{equation}
 in view of (\ref{39}) and (\ref{38}).
\end{proof}
\begin{remark} 
 If we choose $\beta=0$ in Lemma \ref{40} we obtain that $\|DE
 \|_{L^1(S\cap B^S)}=O(h)$.
\end{remark}

We estimate the error $E$ 
outside $B^S$, $B$ as in Lemma \ref{40}, which means de facto in 
$B_2 \setminus B^S$ since $\supp \tilde l\subset B_{\frac{3}{2}r_0}(z_0)$ and get
\begin{equation}
 \begin{aligned}
  \int_{B_2\setminus B^S}|DE| &\le
  \left(\int_{B_2 \setminus B^S}r^{-2}\right)^{\frac{1}{2}}
   \left(\int_{S}r^{2}|DE|^2\right)^{\frac{1}{2}} \\
   &\le   
 c|\log h|^{\frac{1}{2}}
   \left(\int_{S}r^{2}|DE|^2\right)^{\frac{1}{2}}
 \end{aligned}
\end{equation}
and
\begin{equation}
 \begin{aligned}
  \int_{S}r^{2}|DE|^2
   =& \int_{S}\left<DE,D(r^2E)\right>
   -2Er\left<DE,Dr\right> \\
   \le& \int_{S}\left<DE,D(r^2E)\right>
   +2\left(\int_S E^2\right)^{\frac{1}{2}}
   \left(\int_Sr^2|DE|^2\right)^{\frac{1}{2}}
 \end{aligned}
\end{equation}
which leads by Peter-Paul inequality to 
\begin{equation} \label{41}
\begin{aligned}
   \int_{S}r^{2}|DE|^2 \le 2\int_{S}\left<DE,D(r^2E)\right> +4 \int_SE^2 
  \le 2\int_{S}\left<DE,D(r^2E)\right> + ch^2
   \end{aligned}
\end{equation}
in view of Remark \ref{44} (ii). The next goal is to show
\begin{equation}\label{48}
\int_{S}\left<DE,D(r^2E)\right>\le \frac{1}{4}\int_Sr^2|DE|^2+ch^2|\log h|
\end{equation}
which implies (\ref{42}).
Define 
\begin{equation}
 T^1 =\{ \tau \in T_h: \dist(\tilde z_0, \tau )\ge h \}, \quad \Omega_1 = \bigcup_{\tau\in T^1}\tau
 \subset S_h.
\end{equation}
and note, that for small $h$
\begin{equation}
 \{z\in S: \text{dist}_S(z, z_0)\ge 3h\}\subset \Omega_1^S
\end{equation}
and $\tilde l\in C^{\infty}(\Omega_1^S)$. Let $\tilde l_I$ be a function in $V_h$
which equals $\tilde l$ at all nodes in $\Omega_1$.
Let $\bar l_I$ denote the lift of $\tilde l_I$ to $S$.
\begin{lemma} \label{43}
 There hold
 \begin{equation} 
  \begin{aligned}
   \int_{\Omega_1^S}(\bar l_I-\tilde l)^2 \le c h^2, \quad 
  \int_{\Omega_1^S}r^{-2}|D(r^2(\bar l_I-\tilde l))|^2 \le c h^2 |\log h|.
   \end{aligned}
 \end{equation}
\end{lemma}
\begin{proof}
 Let $\tau \in T^1$ then 
 \begin{equation}
  \begin{aligned}
  \|\tilde l - \bar l_I\|_{W^{s, \infty}(\tau^S)}
  \le& c h^{2-s}\|\tilde l\|_{W^{2, \infty}(\tau^S)}\\
  \le& c h^{2-s}(\min_{\tau}r)^{-2}\quad s=0, 1.
  \end{aligned}
  \end{equation}
Since $\min_{\tau}r\ge h$, $\max_{\tau}r-\min_{\tau}r\le h$ we have
\begin{equation} \label{46}
 \frac{\max_{\tau}r}{\min_{\tau}r}\le 2
\end{equation}
and hence for $\beta\ge 0$
\begin{equation}
\begin{aligned}
 \int_{\tau^S}r^{\beta}(\tilde l-\bar l_I)^2+h^2\int_{\tau^S}r^{\beta}|D(\tilde l-\bar l_h)|^2
 \le c\int_{\tau^S}r^{\beta}h^4(\min_{\tau}r)^{-4} 
\le c \int_{\tau^S}r^{\beta-4}h^4.
\end{aligned}
 \end{equation}
 Summing over all $\tau \in T^1$ implies the Lemma since
 \begin{equation}
  \int_{\tau^S}r^{\beta-4}h^4 \le \begin{cases}
                                   c h^{\beta+2}, \quad &\text{ if }\beta <2,\\
                                   c |\log h|h^4, \quad &\text{ if }\beta =2
                                  \end{cases}
 \end{equation}
and
\begin{equation}
r^{-2}|D(r^2(\tilde l-\bar l_I))|^2 \le 8 (\tilde l- \bar l_I)^2+2r^2|D(\tilde l-\bar l_I)|^2.
\end{equation}
\end{proof}

We conclude
\begin{equation} \label{60}
\begin{aligned}
\|\tilde l_I-\tilde l_h\|_{L^{\infty}(\Omega_1)} \le& c h^{-1}\|\tilde l_I-\tilde l_h\|_{L^{2}(\Omega_1)}\\
\le& c h^{-1}(\|\bar l_I-\tilde l\|_{L^{2}(\Omega^S_1)}+\|\tilde L_h-\tilde l\|_{L^{2}(\Omega^S_1)}) \\
\le& c
\end{aligned}
\end{equation}
in view of Lemma \ref{43} and Remark \ref{44} (ii).
\begin{lemma} \label{51}
Let $\varphi \in V_h$ and $v=(r^2\varphi)_I\in V_h$ the linear interpolation of $r^2\varphi$ 
in $\Omega_1$ then
\begin{equation}
\int_{\Omega_1}r^{-2}|D(r^2\varphi-v)|^2 \le c \int_{\Omega_1}\varphi^2
\end{equation}
\end{lemma}
\begin{proof}
For $\tau \in T^1$ we have
\begin{equation} \label{47}
\begin{aligned}
|r^2\varphi-v|_{W^{1, \infty}(\tau)} \le& ch |r^2\varphi|_{W^{2, \infty}(\tau)} \\
\le& c h \sum_{j=1}^2|r^2|_{W^{j, \infty}(\tau)}|\varphi|_{W^{2-j, \infty}(\tau)}
\end{aligned}
\end{equation}
because $D^2(\varphi|\tau)=0$. In view of (\ref{46}) and $r\ge h$ on $\Omega_1$ there holds
\begin{equation}
|r^2|_{W^{j, \infty}(\tau)} \le c \inf_{\tau}r^{2-j} \le c \inf_{\tau}rh^{1-j}
\end{equation}
and in view of an inverse estimate
\begin{equation}
\begin{aligned}
|\varphi|_{W^{2-j, \infty}(\tau)} \le& c h^{j-3}\|\varphi\|_{L^2(\tau)}.
\end{aligned}
\end{equation}
Applying these estimates in (\ref{47}) gives
\begin{equation} \label{47_}
\begin{aligned}
|r^2\varphi-v|_{W^{1, \infty}(\tau)} \le& 
c h^{-1}\inf_{\tau}r\|\varphi\|_{L^2(\tau)}
\end{aligned}
\end{equation}
which leads to
\begin{equation}
\int_{\tau}r^{-2}|D(r^2\varphi-v)|^2 \le c \int_{\tau}\varphi^2 
\end{equation}
by estimating the integrand in the $L^{\infty}$-norm. Summing over $\tau \in T^1$ 
gives the claim.
\end{proof}

\begin{lemma}
Estimate (\ref{48}) holds.
\end{lemma}
\begin{proof}
For $v_h\in V_h$ we have
\begin{equation}
a(E, V_h) = O(h^2)\|\tilde L_h\|_{H^1(S)}\|V_h\|_{H^1(S)}
\end{equation}
and estimate
\begin{equation} \label{52}
\begin{aligned}
\int_{S}\left<DE, D(r^2E)\right> =& \int_S\left<DE, D(r^2E-V_h)\right>-\int_SEV_h \\
&+ O(h^2)\|\tilde L_h\|_{H^1(S)}\|V_h\|_{H^1(S)}\\
\stackrel{Lemma\ \ref{40}, Remark\ \ref{44}}{\le}& \int_{\Omega_1}\left<DE, D(r^2E-V_h)\right> 
+ c(h^2+h|V_h|_{W^{1,\infty}(S\setminus \Omega_1)}) \\
&+\int_S V_h^2+O(h^2)\|\tilde L_h\|_{H^1(S)}\|V_h\|_{H^1(S)}.
\end{aligned}
\end{equation}
If $v_h$ interpolates $r^2(\bar l_I-\tilde L_h)$ in $\Omega_1$ 
then
\begin{equation} \label{53}
 \begin{aligned}
  \int_{\Omega_1^S}\left<DE, D(r^2E-V_h)\right>
  \le& \frac{1}{16}\int_Sr^2|DE|^2+4\int_{\Omega_1^S}r^{-2}|D(r^2E-V_h)|^2 \\
  \le& \frac{1}{16}\int_S r^2|DE|^2+8\int_{\Omega_1^S}r^{-2}|D(r^2(\tilde l-\bar l_I))|^2\\
 & + 8\int_{\Omega_1^S}r^{-2}|D(r^2(\bar l_I-\tilde L_h)-V_h)|^2 \\
 \stackrel{Lemma\ \ref{43}(ii), Lemma\ \ref{51}}{\le}&\frac{1}{16}
 \int_S r^2|DE|^2+ch^2|\log h|+c\int_{\Omega_1^S}(\bar l_I-\tilde L_h)^2 \\
 \le& \frac{1}{16}\int_S r^2|DE|^2+ch^2|\log h|\\
 &+
 c \left(\int_{\Omega_1^S}(\bar l_I-\tilde l)^2
 +\int_{\Omega_1^S}(\tilde l-\tilde L_h)^2\right)\\
 \stackrel{Lemma\ \ref{43}, Remark\ \ref{44}}{\le}&
 \frac{1}{16}\int_{S}r^2|DE|^2+ch^2|\log h|.
 \end{aligned}
\end{equation}
We use (\ref{53}) to estimate the first summand on the right-hand side of (\ref{52})
and obtain
\begin{equation}
\begin{aligned}
 \int_{S}\left<DE, D(r^2E)\right>\le& c h^2|\log h| +\frac{1}{16}\int_Sr^2|DE|^2
 +ch|V_h|_{W^{1,\infty}(S\setminus \Omega_1^S)} +\int_S V_h^2 \\
 &+ O(h^2)\|\tilde L_h\|_{H^1(S)}\|V_h\|_{H^1(S)}
 \end{aligned}
\end{equation}
We estimate $v_h$ with standard interpolation estimates
\begin{equation}
 \begin{aligned}
  h|v_h|_{W^{1, \infty}(S_h\setminus \Omega_1)}
  +h^{-1}\|v_h\|_{L^2(S_h\setminus \Omega_1)} \le& 
  c \sup_{S\setminus \Omega_1^S}|r^2(\bar l_I-\tilde L_h)|\\
  =& \sup_{\partial \Omega_1^S}|r^2(\bar l_I-\tilde L_h)|\\
  \le& ch^2
 \end{aligned}
\end{equation}
where we assume w.l.o.g. that
$v_h$ is zero at all nodes in the interior of 
$S_h\setminus\Omega_1$ and for the last inequality estimate (\ref{60}). Furthermore, we have
\begin{equation}
\begin{aligned}
 \|v_h\|_{L^2(\Omega_1)} \le& c \|r^2(\bar l_I-\tilde L_h)\|_{L^2(\Omega_1^S)}\\
 \le& c \|\bar l_I-\tilde l\|_{L^2(\Omega_1^S)}+\|\tilde l-\tilde L_h\|_{L^2(\Omega_1^S)}\\
 \le& ch
\end{aligned}
 \end{equation}
in view of Lemma \ref{43} and Remark \ref{44} and
\begin{equation}
\|V_h\|_{H^1(S)} \le c h^{-1}\|V_h\|_{L^2(S)} \le ch
\end{equation}
and
\begin{equation}
\|\tilde L_h\|_{H^1(S)} \le c h^{-1}\|\tilde L_h\|_{L^2(S)} \le c.
\end{equation}
 \end{proof}

\end{document}